\documentclass{article}

\usepackage{graphicx} 
\graphicspath{{Images/}}
\usepackage{comment}
\usepackage{amsmath, amssymb} 
\usepackage{amsthm} 
\usepackage{tikz} 
\usetikzlibrary{shapes.geometric} 
\usetikzlibrary{calc} 
\usepackage[backend=biber, style=numeric, sorting=nyt]{biblatex} 
\usepackage{microtype} 
\usepackage{hyperref} 
\usepackage{subcaption} 
\usepackage{wrapfig} 

\usepackage{subfiles} 

\newtheorem{thm}{Theorem}[section]
\newtheorem{prop}[thm]{Proposition}
\theoremstyle{definition}
\newtheorem{defn}[thm]{Definition}
\newtheorem{example}[thm]{Example}

\title{Constructing Block Designs from Complete Graphs}
\author{Benjamin Glancy and Leanne Holder}
\date{2026 May}

\begin{document}

\maketitle

\begin{abstract}
Block designs are combinatorial structures in which each pair of a set of varieties appears together in a fixed number of blocks. Complete graphs are graphs in which every pair of vertices are adjacent. We present some new constructions of block designs using complete graphs, including two infinite families of designs using edge sets of complete graphs.
\end{abstract}
\vfill


\newpage
\section{Introduction}

A block design, introduced by F. Yates in 1936 \cite{eugenics}, is an incidence structure consisting of a set of objects called varieties and a family of subsets known as  blocks. Block designs were first used for statistical experimental design. They can be used to test several different treatments simultaneously in a controlled way. For example, a biologist may use a block design to administer several different medications to a set of mice, or a computer engineer may use a block design to test the effectiveness of upgrading several components of a computer system. Beyond these applications, design theory has become a rich area of mathematical research. Block designs are closely related to many other objects of combinatorial study, including finite fields, Latin squares, and graphs.

Graphs are another versatile type of structure. Graphs encode the connections among a set of objects. They have applications spanning from cartography to machine learning. Within mathematics, graphs appear in the study of groups, fractals, and of course, block designs.

In this paper we create two families of block designs from complete graphs with an odd number of vertices.

\subsection{Balanced Incomplete Block Designs}

\begin{defn}
A \emph{balanced incomplete block design (BIBD)} with parameters $(v, b, r, k, \lambda)$ is an ordered pair $(V, B)$ of a set $V$ of \emph{varieties} and a set $B$ of subsets of $V$ called \emph{blocks} such that $|V| = v$, $|B| = b$, each variety is an element of exactly $r$ blocks, each block has size $k$ (with $k < v$), and each pair of varieties is a subset of exactly $\lambda$ blocks.
\end{defn}

The parameter $r$ is called the \emph{replication number} and  $\lambda$ is called the \emph{index}. The five parameters of a BIBD are not all independent. The following proposition gives two fundamental relationships between them.

\begin{prop}\label{prop-bibd-parameters}
Let $(v, b, r, k, \lambda)$ be the parameters of a BIBD. Then
\begin{equation}\label{eq-vr-bk}
v r = b k
\end{equation}
and
\begin{equation}\label{eq-rk-lambdav}
r (k - 1) = \lambda (v - 1).
\end{equation}
\end{prop}

\begin{proof}
For Equation (\ref{eq-vr-bk}), we count the number of ordered pairs $(x, Y)$ with $x$ a variety and $Y$ a block containing $x$. For the left-hand side, note that there are $v$ varieties and each one appears $r$ times. For the right-hand side, note that there are $b$ blocks and each block contains $k$ varieties. Thus $v r = b k$.

For Equation (\ref{eq-rk-lambdav}), we fix a variety $x$ and count the number of ordered pairs $(y, Y)$ with $y$ a variety and $Y$ a block such that $\{ x, y \} \subseteq Y$. For the left-hand side, note that $x$ is an element of $r$ blocks, and each of those blocks contains $k - 1$ other varieties. For the right-hand side, note that there are $v - 1$ varieties besides $x$ and $x$ shares $\lambda$ blocks with each of them. Thus $r (k - 1) = \lambda (v - 1)$. 
\end{proof}

A practical consequence of Proposition \ref{prop-bibd-parameters} is that we do not need to list all five parameters. It is customary to denote a $(v, b, r, k, \lambda)$-BIBD as a $(v, k, \lambda)$-BIBD or a $(v, k, \lambda)$-design. \emph{Balanced} refers to the fact that the index $\lambda$ is constant and \emph{incomplete} means that $k < v$, so no block contains all of the varieties.  A design is considered \emph{symmetric} if the number of varieties equals the number of blocks.  A symmetric design also has $r=k$ which follows from applying Proposition \ref{prop-bibd-parameters}.

Below we give two examples of block designs, one which is symmetric and a second which is not.

\begin{example} {\bf A symmetric (7,3,1)-BIBD} \label{example-fano-plane}

Consider the following variety set $V=\{0,1,2,3,4,5,6\}$ and block set $B = \{ \{ 0, 1, 3 \},\{ 1, 2, 4 \}, \{ 2, 3, 5 \},\{ 3, 4, 6 \},\{ 4, 5, 0 \}, \{ 5, 6, 1 \},\{ 6, 0, 2 \} \}$.  We clearly have $v=7$ varieties.  There are $b=7$ blocks each  with $k=3$ varieties.  It is easy to verify that each element of $V$ appears in $r=3$ blocks and any pair of elements appear in $\lambda =1$ blocks.  This design is balanced because $\lambda$ is the constant 1. Furthermore, this design is both symmetric and incomplete as $v=7=b$ and $k=3< 7=v$. Hence we have a $(7,7,3,3,1)-$ or $(7,3,1)-$balanced incomplete block design. 
\end{example}

\begin{example} {\bf A non-symmetric (16,4,1)-BIBD}

For this $(16, 4, 1)$-design we use the $v=16$ letters $a$ through $p$ as the varieties with the $b = 20$ blocks given by
\begin{align*} 
& \{a, b, c, d\}, && \{e, f, g, h\}, && \{i, j, k, l\}, && \{m, n, o, p\}, && \{a, e, i, m\}, && \{a, f, j, n\}, \\
& \{a, g, k, o\}, && \{a, h, l, p\}, && \{b, e, j, o\}, && \{b, f, i, p\}, && \{b, g, l, m\}, && \{b, h, k, n\}, \\
& \{c, e, l, o\}, && \{c, f, k, p\}, && \{c, g, i, n\}, && \{c, h, j, m\}, && \{d, e, k, n\}, && \{d, f, l, m\}, \\
& \{d, g, j, p\}, && \{d, h, i, o\}.
\end{align*}
Each block contains four letters, so $k = 4$. Each pair of letters appears in one block; for instance, $\{ a, g \} \subseteq \{ a, g, k, o \}$ and $\{ b, i \} \subseteq \{ b, f, i, p \}$. The reader is encouraged to convince themselves that this is true for every pair.  We use Equation (\ref{eq-rk-lambdav}) from Proposition \ref{prop-bibd-parameters} to see that 
\[ r = \frac{\lambda (v - 1)}{k - 1} = \frac{1 \cdot 15}{3} = 5, \]
so every letter is an element of five blocks.
\end{example}

\begin{defn}
A \emph{$t$-design} with parameters $t$-$(v, b, r, k, \lambda_t)$ is an ordered pair $(V, B)$ of a set $V$ of \emph{varieties} and a set $B$ of subsets of $V$ called \emph{blocks} such that $|V| = v$, $|B| = b$, each variety is an element of exactly $r$ blocks, each block has size $k$ (with $k < v$), and each $t$-subset of $V$ is a subset of exactly $\lambda_t$ blocks.
\end{defn}

Every balanced incomplete block design is a $t$-design with $t = 2$.

\begin{example}[\cite{database}]
Here is a $3$-$(10, 4, 1)$ design. The 10 varieties are the elements of $\mathbb{Z}_{10}$ and the $b = 30$ blocks are
\begin{align*}
& \{0, 1, 2, 8\}, && \{0, 1, 3, 6\}, && \{0, 1, 4, 5\}, && \{0, 1, 7, 9\}, && \{0, 2, 3, 7\}, && \{0, 2, 4, 6\}, \\
& \{0, 2, 5, 9\}, && \{0, 3, 4, 9\}, && \{0, 3, 5, 8\}, && \{0, 4, 7, 8\}, && \{0, 5, 6, 7\}, && \{0, 6, 8, 9\}, \\
& \{1, 2, 3, 4\}, && \{1, 2, 5, 7\}, && \{1, 2, 6, 9\}, && \{1, 3, 5, 9\}, && \{1, 3, 7, 8\}, && \{1, 4, 6, 7\}, \\
& \{1, 4, 8, 9\}, && \{1, 5, 6, 8\}, && \{2, 3, 5, 6\}, && \{2, 3, 8, 9\}, && \{2, 4, 5, 8\}, && \{2, 4, 7, 9\}, \\
& \{2, 6, 7, 8\}, && \{3, 4, 5, 7\}, && \{3, 4, 6, 8\}, && \{3, 6, 7, 9\}, && \{4, 5, 6, 9\}, && \{5, 7, 8, 9\}.
\end{align*}
Since $t = 3$ for this design, the parameter $\lambda = 1$ means that every set of three varieties is a subset of exactly one block.
\end{example}

\subsection{Complete Graphs}

\begin{defn}
A \emph{graph} is an ordered pair $(V, E)$ of a finite set $V$ of \emph{vertices} and a symmetric, anti-reflexive relation $E$ on $V$ called the \emph{edge set}. An element of $E$ is called an \emph{edge} of the graph.
\end{defn}

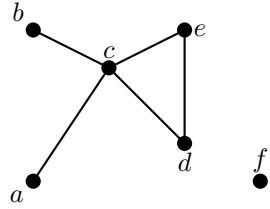
\begin{figure}[h]
\centering
\begin{tikzpicture}
\fill[black] (0, 0) circle[radius=0.1];
\fill[black] (0, 2) circle[radius=0.1];
\fill[black] (1, 1.5) circle[radius=0.1];
\fill[black] (2, 0.5) circle[radius=0.1];
\fill[black] (2, 2) circle[radius=0.1];
\fill[black] (3, 0) circle[radius=0.1];
\draw[thick] (0, 0) -- (1, 1.5) -- (0, 2);
\draw[thick] (2, 2) -- (1, 1.5) -- (2, 0.5) -- (2, 2);
\node[anchor=north east] at (0, 0) {$a$};
\node[anchor=south east] at (0, 2) {$b$};
\node[anchor=south] at (1, 1.5) {$c$};
\node[anchor=north] at (2, 0.5) {$d$};
\node[anchor=west] at (2, 2) {$e$};
\node[anchor=south] at (3, 0) {$f$};
\end{tikzpicture}
\caption{A garden-variety graph\label{fig-generic-graph}}
\end{figure}

Graphs are typically represented visually, as in Figure \ref{fig-generic-graph}. The vertices are shown as dots and the relation $E$ is shown as a set of lines connecting the dots. For example, the vertices $a$ and $c$ in Figure \ref{fig-generic-graph} are connected by a line, so we know both $(a, c)$ and $(c, a)$ are in $E$.

\begin{defn}
In a graph $(V, E)$, two vertices $x, y \in V$ are \emph{adjacent} if $(x, y) \in E$ and we call $y$ a \emph{neighbor} of $x$ (and $x$ a neighbor of $y$).  Furthermore, the \emph{degree} of a vertex $x$ is the number of neighbors of $x$.
\end{defn}

To return to the example of Figure \ref{fig-generic-graph}, vertices $a$ and $c$ are adjacent but $a$ and $b$ are not. The degree of vertex $c$ is four because it has four neighbors: $a$, $b$, $d$, and $e$. Vertex $f$ has degree zero because it is not adjacent to any other vertices.

\begin{defn}
A graph $(V, E)$ is \emph{regular} of degree $k$ if every vertex in $V$ has degree exactly $k$.
\end{defn}

\begin{figure}[h]
\begin{subfigure}{0.5\linewidth}
\centering
\begin{tikzpicture}
\fill[black] (0.75, 0) circle[radius=0.1];
\fill[black] (2.25, 0) circle[radius=0.1];
\fill[black] (0, 1.5) circle[radius=0.1];
\fill[black] (0.75, 3) circle[radius=0.1];
\fill[black] (2.25, 3) circle[radius=0.1];
\fill[black] (3, 1.5) circle[radius=0.1];
\draw[thick] (0.75, 0) -- (2.25, 0) -- (3, 1.5) -- (2.25, 3) -- (0.75, 3) -- (0, 1.5) -- (0.75, 0);
\node[anchor=south east] at (0.75, 3) {1};
\node[anchor=south west] at (2.25, 3) {2};
\node[anchor=west] at (3, 1.5) {$\, 3$};
\node[anchor=north west] at (2.25, 0) {4};
\node[anchor=north east] at (0.75, 0) {5};
\node[anchor=east] at (0, 1.5) {$6 \,$};
\end{tikzpicture}
\caption{The cycle graph of 6 vertices\label{fig-c6}}
\end{subfigure}
\begin{subfigure}{0.5\linewidth}
\centering
\begin{tikzpicture}
\fill[black] (0.75, 0) circle[radius=0.1];
\fill[black] (2.25, 0) circle[radius=0.1];
\fill[black] (0, 1.5) circle[radius=0.1];
\fill[black] (0.75, 3) circle[radius=0.1];
\fill[black] (2.25, 3) circle[radius=0.1];
\fill[black] (3, 1.5) circle[radius=0.1];
\draw[thick] (0.75, 0) -- (2.25, 0) -- (3, 1.5) -- (2.25, 3) -- (0.75, 3) -- (0, 1.5) -- (0.75, 0);
\draw[thick] (0.75, 3) -- (3, 1.5) -- (0.75, 0) -- (0.75, 3);
\draw[thick] (2.25, 3) -- (2.25, 0) -- (0, 1.5) -- (2.25, 3);
\draw[thick] (0, 1.5) -- (3, 1.5) (0.75, 3) -- (2.25, 0) (0.75, 0) -- (2.25, 3);
\node[anchor=south east] at (0.75, 3) {1};
\node[anchor=south west] at (2.25, 3) {2};
\node[anchor=west] at (3, 1.5) {$\, 3$};
\node[anchor=north west] at (2.25, 0) {4};
\node[anchor=north east] at (0.75, 0) {5};
\node[anchor=east] at (0, 1.5) {$6 \,$};
\end{tikzpicture}
\caption{The complete graph of 6 vertices\label{fig-k6}}
\end{subfigure}
\caption{\label{fig-reg-graph-examples}}
\end{figure}

\begin{example}
Both graphs in Figure \ref{fig-reg-graph-examples} are regular. Figure \ref{fig-c6} is an example of a \emph{cycle graph}, an infinite family of graphs consisting of a single loop that connects all of the vertices. All cycle graphs are regular with $k = 2$. Figure \ref{fig-k6} is an example of a \emph{complete graph}, a graph in which every vertex is adjacent to every other vertex. The complete graph of $n$ vertices is denoted $K_n$. Every vertex of $K_n$ has degree $n - 1$.
\end{example}

Complete graphs are the foundation of our work to create new block designs. The next proposition illustrates the connection between graphs and designs and serves to inspire our work.

\begin{prop}[\cite{cameron-lint}, \cite{lint-wilson}] \label{prop-cam-lint-ex2.2}
Let $E$ be the edge set of $K_5$ and let $B$ be the set of subgraphs of $K_5$ of the three types shown in Figure \ref{fig-ex2.2-k5-blocks}. Then $(E, B)$ is a $3$-$(10, 4, 1)$ $t$-design.
\end{prop}

\begin{figure}[h]
\begin{subfigure}{0.3\linewidth}
\begin{tikzpicture}
\fill[black] (0.8, 0) circle[radius=0.1];
\fill[black] (2.4, 0) circle[radius=0.1];
\fill[black] (0, 1.6) circle[radius=0.1];
\fill[black] (3.2, 1.6) circle[radius=0.1];
\fill[black] (1.6, 2.8) circle[radius=0.1];
\draw (1.6, 2.8) -- (0, 1.6);
\draw (1.6, 2.8) -- (0.8, 0);
\draw (1.6, 2.8) -- (2.4, 0);
\draw (1.6, 2.8) -- (3.2, 1.6);
\end{tikzpicture}
\caption{``fan''}
\end{subfigure}
\hfill
\begin{subfigure}{0.3\linewidth}
\begin{tikzpicture}
\fill[black] (0.8, 0) circle[radius=0.1];
\fill[black] (2.4, 0) circle[radius=0.1];
\fill[black] (0, 1.6) circle[radius=0.1];
\fill[black] (3.2, 1.6) circle[radius=0.1];
\fill[black] (1.6, 2.8) circle[radius=0.1];
\draw (0, 1.6) -- (3.2, 1.6) -- (2.4, 0) -- (0.8, 0) -- (0, 1.6);
\end{tikzpicture}
\caption{``rectangle''}
\end{subfigure}
\hfill
\begin{subfigure}{0.3\linewidth}
\begin{tikzpicture}
\fill[black] (0.8, 0) circle[radius=0.1];
\fill[black] (2.4, 0) circle[radius=0.1];
\fill[black] (0, 1.6) circle[radius=0.1];
\fill[black] (3.2, 1.6) circle[radius=0.1];
\fill[black] (1.6, 2.8) circle[radius=0.1];
\draw (0, 1.6) -- (1.6, 2.8) -- (3.2, 1.6) -- (0, 1.6);
\draw (0.8, 0) -- (2.4, 0);
\end{tikzpicture}
\caption{``triangle''}
\end{subfigure}
\caption{Blocks of a $3$-$(10, 4, 1)$ design}
\label{fig-ex2.2-k5-blocks}
\end{figure}
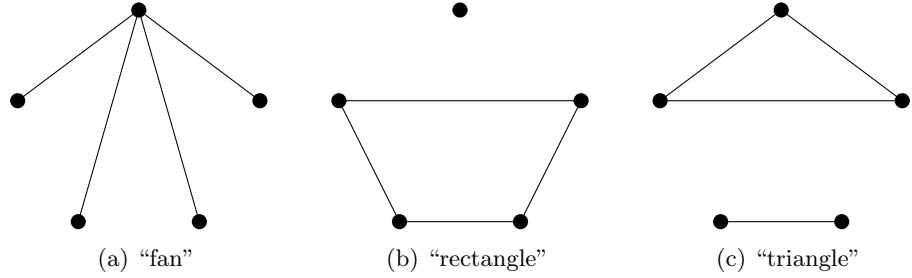

\begin{proof}
In the graph $K_5$, every pair of the five vertices are adjacent so there are ${5 \choose 2} = 10$ edges. Thus $v = |E| = 10$. Each of the block types shown in Figure \ref{fig-ex2.2-k5-blocks} has four edges, so $k = 4$.

We claim that $t = 3$ and $\lambda = 1$, that is, every set of three edges is a subset of one block. There are four cases for the arrangement of three edges:
\begin{itemize}
\item[(a)] all adjacent at a single vertex (Figure \ref{fig-k5-subfan}),
\item[(b)] connected end-to-end in a path (Figure \ref{fig-k5-path}),
\item[(c)] forming a 3-cycle (Figure \ref{fig-k5-triangle}),
\item[(d)] two edges adjacent and the third disconnected (Figure \ref{fig-k5-disconnected}).
\end{itemize}
Case (a) appears in one fan-type block. Case (b) appears in one rectangle-type block. Cases (c) and (d) each appear in one triangle-type block. Since each 3-subset appears in exactly one block, we have $\lambda = 1$.

\begin{figure}[h]
\centering
\begin{subfigure}{0.24\linewidth}
\centering
\begin{tikzpicture}
\fill[black] (0.6, 0) circle[radius=0.1];
\fill[black] (1.8, 0) circle[radius=0.1];
\fill[black] (0, 1.2) circle[radius=0.1];
\fill[black] (2.4, 1.2) circle[radius=0.1];
\fill[black] (1.2, 2.1) circle[radius=0.1];
\draw (1.2, 2.1) -- (0, 1.2);
\draw (1.2, 2.1) -- (0.6, 0);
\draw (1.2, 2.1) -- (1.8, 0);
\end{tikzpicture}
\caption{\label{fig-k5-subfan}}
\end{subfigure}
\hfill
\begin{subfigure}{0.24\linewidth}
\centering
\begin{tikzpicture}
\fill[black] (0.6, 0) circle[radius=0.1];
\fill[black] (1.8, 0) circle[radius=0.1];
\fill[black] (0, 1.2) circle[radius=0.1];
\fill[black] (2.4, 1.2) circle[radius=0.1];
\fill[black] (1.2, 2.1) circle[radius=0.1];
\draw (0, 1.2) -- (2.4, 1.2) -- (1.8, 0) -- (0.6, 0);
\end{tikzpicture}
\caption{\label{fig-k5-path}}
\end{subfigure}
\hfill
\begin{subfigure}{0.24\linewidth}
\centering
\begin{tikzpicture}
\fill[black] (0.6, 0) circle[radius=0.1];
\fill[black] (1.8, 0) circle[radius=0.1];
\fill[black] (0, 1.2) circle[radius=0.1];
\fill[black] (2.4, 1.2) circle[radius=0.1];
\fill[black] (1.2, 2.1) circle[radius=0.1];
\draw (0, 1.2) -- (1.2, 2.1) -- (2.4, 1.2) -- (0, 1.2);
\end{tikzpicture}
\caption{\label{fig-k5-triangle}}
\end{subfigure}
\hfill
\begin{subfigure}{0.24\linewidth}
\centering
\begin{tikzpicture}
\fill[black] (0.6, 0) circle[radius=0.1];
\fill[black] (1.8, 0) circle[radius=0.1];
\fill[black] (0, 1.2) circle[radius=0.1];
\fill[black] (2.4, 1.2) circle[radius=0.1];
\fill[black] (1.2, 2.1) circle[radius=0.1];
\draw (0, 1.2) -- (1.2, 2.1) -- (2.4, 1.2);
\draw (0.6, 0) -- (1.8, 0);
\end{tikzpicture}
\caption{\label{fig-k5-disconnected}}
\end{subfigure}
\caption{\label{fig-3-edges-k5}}
\end{figure}
\end{proof}

\section{Designs on Complete Graphs}

We now present the first of two infinite families of BIBDs built on the edge sets of complete graphs. Theorem \ref{thm-kp-designs} describes designs in which the blocks are the edge sets of \emph{paths}, graphs that consist of several edges connected end-to-end. The path with $n$ vertices ($n - 1$ edges) is denoted by $P_n$. 

\begin{thm} \label{thm-kp-designs}
Let $E$ be the edge set of the complete graph on $n$ vertices, $K_n$, and let $B$ be the set of path subgraphs of $K_n$ of length $\hat{k}$ (that is, $\hat{k}$ edges, or $\hat{k} + 1$ vertices) with $\hat{k} \geq 2$. The pair $(E, B)$ forms a balanced incomplete block design, referred to as a KP design, if and only if $n = 2 \hat{k} - 1$. In that case, $(E, B)$ is a block design with parameters
\begin{align*}
v & = {2 \hat{k} - 1 \choose 2}, \\
b & = \frac{1}{2} P(2 \hat{k} - 1, \hat{k} + 1) = \frac{(2 \hat{k} - 1)!}{2 \cdot (\hat{k} - 2)!}, \\
r & = \hat{k} \cdot P(2 \hat{k} - 3, \hat{k} - 1) = \frac{\hat{k} \cdot (2 \hat{k} - 3)!}{(\hat{k} - 2)!}, \\
k & = \hat{k}, \\
\lambda & = (\hat{k} - 1) \cdot P(2 \hat{k} - 4, \hat{k} - 2) = \frac{(\hat{k} - 1) \cdot (2 \hat{k} - 4)!}{(\hat{k} - 2)!}.
\end{align*}
\end{thm}

\begin{proof}
First, assume that $n = 2 \hat{k} - 1$. The number of varieties, $v$, in the design is the number of edges in the graph $K_{2\hat{k}-1}$. Since each pair of vertices have an edge between them, we have
\[ v = {n \choose 2} = {2 \hat{k} - 1 \choose 2}. \]
The blocks of the design are defined to be paths of length $\hat{k}$, so the size of each block is $k = \hat{k}$.

To count the total number of blocks $b$, we consider each block as a string of $\hat{k} + 1$ vertices, chosen from the set of $2 \hat{k} - 1$ total vertices. However, a path subgraph has no direction, so a string of vertices is equivalent forwards and backwards. Thus the total number of blocks is
\[ b = \frac{1}{2} P (n, \hat{k} + 1) = \frac{1}{2} P(2 \hat{k} - 1, \hat{k} + 1). \]
Expanding this expression in terms of factorials produces
\[ b = \frac{(2 \hat{k} - 1)!}{2 \cdot (\hat{k} - 2)!}. \]

We use a similar approach to count the replication number $r$. We denote the edge between the two vertices $a$ and $b$ by $(a, b)$. Furthermore, blocks that contain $(a, b)$ correspond to strings of vertices that contain the substring $a b$. There are $\hat{k}$ choices for the position of $a b$ in the string and $P(2 \hat{k} - 3, \hat{k} - 1)$ ways to fill in the remaining $\hat{k} - 1$ required vertices. Note that we no longer need to divide by 2 because requiring the substring $a b$ imposes a direction on the path; for example, this method counts the string $a b c$ but not the equivalent string $c b a$ because the latter includes $b a$ rather than $a b$. Thus we have
\[ r = \hat{k} \cdot P(2 \hat{k} - 3, \hat{k} - 1) = \frac{\hat{k} \cdot (2 \hat{k} - 3)!}{(\hat{k} - 2)!}. \]

The final parameter, $\lambda$, is the guarantee that this construction is truly a balanced incomplete block design. We claim that every pair of edges appears together in $(\hat{k} - 1) \cdot P(2 \hat{k} - 4, \hat{k} - 2)$ blocks. Two edges can either be adjacent, meaning they share one vertex, or nonadjacent.

A pair of adjacent edges is a path of length 2, which is represented by a string of length 3. Without loss of generality, consider the edges $(a, b)$ and $(b, c)$. The blocks that contain both of these edges can be represented by strings of length $\hat{k} + 1$ that contain the substring $a b c$. That substring can appear in any of $\hat{k} - 1$ positions, and the rest of the string can be filled out in $P(2 \hat{k} - 4, \hat{k} - 2)$ ways from the remaining vertices. Thus pairs of adjacent edges appear together in $(\hat{k} - 1) \cdot P(2 \hat{k} - 4, \hat{k} - 2)$ blocks.

Without loss of generality, consider the two nonadjacent edges $(a, b)$ and $(c, d)$. As before, we count strings of length $\hat{k} + 1$ with these substrings. However, in this case of nonadjacent edges, we must account for the relative orders of the two substrings. There are 4 different possible orders: $a b$ or $b a$ for the first, and $c d$ or $d c$ for the second. Next we choose two of the $\hat{k} - 1$ effective positions in the string for the given edges. Now four vertices ($a$, $b$, $c$, and $d$) have been placed, so we permute the remaining $2 \hat{k} - 5$ of them into the remaining $\hat{k} - 3$ positions. All together, the edges $(a, b)$ and $(c, d)$ appear together in
\[ 4 \cdot {\hat{k} - 1 \choose 2} \cdot P(2 \hat{k} - 5, \hat{k} - 3) \]
blocks. It remains to show that this is the same quantity as $\hat{k} \cdot P(2 \hat{k} - 3, \hat{k} - 1)$. In terms of factorials, we have
\begin{align*}
4 \cdot {\hat{k} - 1 \choose 2} \cdot P(2 \hat{k} - 5, \hat{k} - 3) & = \frac{4 (\hat{k} - 1)! (2 \hat{k} - 5)!}{2 (\hat{k} - 3)! (\hat{k} - 2)!} \\
& = \frac{2 (\hat{k} - 1) (\hat{k} - 2) (2 \hat{k} - 5)!}{(\hat{k} - 2)!} \\
& = \frac{(\hat{k} - 1) (2 \hat{k} - 4) (2 \hat{k} - 5)!}{(\hat{k} - 2)!} \\
& = \frac{(\hat{k} - 1) (2 \hat{k} - 4)!}{(\hat{k} - 2)!} \\
& = (\hat{k} - 1) \cdot P(2 \hat{k} - 4, \hat{k} - 2).
\end{align*}
Therefore any pair of vertices appears together in the same number of blocks, so the construction described is indeed a balanced incomplete block design with the stated parameters.

Now suppose that $n \neq 2 \hat{k} - 1$. As above, a pair of adjacent edges share
\[ (\hat{k} - 1) \cdot P(n - 3, \hat{k} - 2) \]
blocks and a pair of nonadjacent edges share
\[ 4 \cdot {\hat{k} - 1 \choose 2} \cdot P(n - 4, \hat{k} - 3) \]
blocks. We expand again in terms of factorials to obtain
\[ \frac{(\hat{k} - 1) (n - 3)!}{(n - \hat{k} - 1)!} \]
for adjacent edges and
\[ \frac{2 (\hat{k} - 1) (\hat{k} - 2) (n - 4)!}{(n - \hat{k} - 1)!} \]
for nonadjacent edges. Unlike before, $n - 3 \neq 2 (\hat{k} - 2)$, so these pairs of adjacent edges and pairs of nonadjacent edges do not appear in the same number of blocks. For $n < 2 \hat{k} - 1$, adjacent edges share fewer blocks, and for $n > 2 \hat{k} - 1$, adjacent edges share more blocks.
\end{proof}

The KP designs are named for the complete graph $K_n$ that provides the variety set and the path graphs $P_{k+1}$ that form the blocks. Let us examine an example.

\begin{example} \label{example-path-design}
Let $k = 3$, so we have paths of 3 edges in a complete graph of $2 \cdot 3 - 1 = 5$ vertices. The remaining parameters of this KP design are
\begin{align*}
v & = {5 \choose 2} = 10, \\
b & = \frac{1}{2} \cdot P(5, 4) = 60, \\
r & = 3 \cdot P(3, 2) = 18, \\
\lambda & = 4.
\end{align*}

Figure \ref{fig-kp3-adjacent-blocks} shows the four blocks containing a pair of adjacent edges. Either of the remaining two vertices can be connected to either end of the two edges in question.

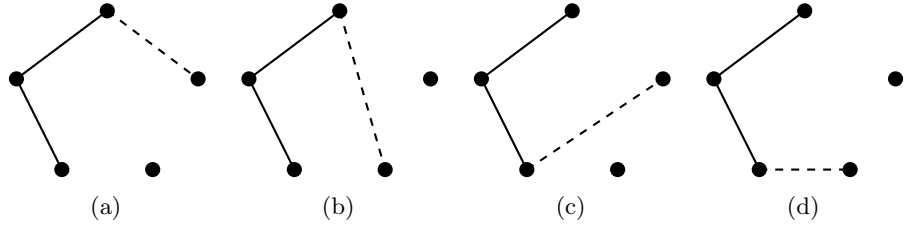
\begin{figure}[h]
\centering
\begin{subfigure}{0.24\linewidth}
\centering
\begin{tikzpicture}
\fill[black] (0.6, 0) circle[radius=0.1];
\fill[black] (1.8, 0) circle[radius=0.1];
\fill[black] (0, 1.2) circle[radius=0.1];
\fill[black] (2.4, 1.2) circle[radius=0.1];
\fill[black] (1.2, 2.1) circle[radius=0.1];
\draw[thick] (0.6, 0) -- (0, 1.2) -- (1.2, 2.1);
\draw[thick, dashed] (1.2, 2.1) -- (2.4, 1.2);
\end{tikzpicture}
\caption{}
\end{subfigure}
\hfill
\begin{subfigure}{0.24\linewidth}
\centering
\begin{tikzpicture}
\fill[black] (0.6, 0) circle[radius=0.1];
\fill[black] (1.8, 0) circle[radius=0.1];
\fill[black] (0, 1.2) circle[radius=0.1];
\fill[black] (2.4, 1.2) circle[radius=0.1];
\fill[black] (1.2, 2.1) circle[radius=0.1];
\draw[thick] (0.6, 0) -- (0, 1.2) -- (1.2, 2.1);
\draw[thick, dashed] (1.2, 2.1) -- (1.8, 0);
\end{tikzpicture}
\caption{}
\end{subfigure}
\hfill
\begin{subfigure}{0.24\linewidth}
\centering
\begin{tikzpicture}
\fill[black] (0.6, 0) circle[radius=0.1];
\fill[black] (1.8, 0) circle[radius=0.1];
\fill[black] (0, 1.2) circle[radius=0.1];
\fill[black] (2.4, 1.2) circle[radius=0.1];
\fill[black] (1.2, 2.1) circle[radius=0.1];
\draw[thick] (0.6, 0) -- (0, 1.2) -- (1.2, 2.1);
\draw[thick, dashed] (0.6, 0) -- (2.4, 1.2);
\end{tikzpicture}
\caption{}
\end{subfigure}
\hfill
\begin{subfigure}{0.24\linewidth}
\centering
\begin{tikzpicture}
\fill[black] (0.6, 0) circle[radius=0.1];
\fill[black] (1.8, 0) circle[radius=0.1];
\fill[black] (0, 1.2) circle[radius=0.1];
\fill[black] (2.4, 1.2) circle[radius=0.1];
\fill[black] (1.2, 2.1) circle[radius=0.1];
\draw[thick] (0.6, 0) -- (0, 1.2) -- (1.2, 2.1);
\draw[thick, dashed] (0.6, 0) -- (1.8, 0);
\end{tikzpicture}
\caption{}
\end{subfigure}
\caption{Four paths containing two adjacent edges\label{fig-kp3-adjacent-blocks}}
\end{figure}

Figure \ref{fig-kp3-nonadjacent-blocks} shows the four blocks containing a pair of nonadjacent edges. Either of the vertices of one edge can be connected to either of the vertices of the other edge.

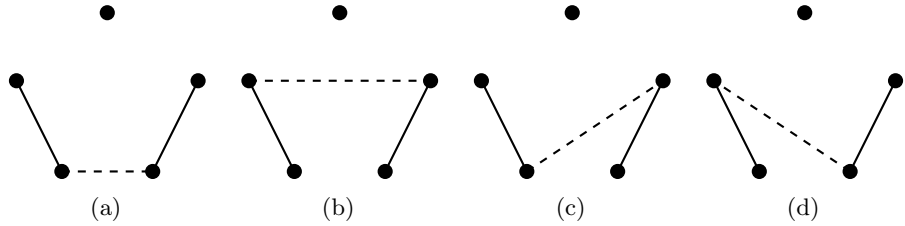
\begin{figure}[h]
\centering
\begin{subfigure}{0.24\linewidth}
\centering
\begin{tikzpicture}
\fill[black] (0.6, 0) circle[radius=0.1];
\fill[black] (1.8, 0) circle[radius=0.1];
\fill[black] (0, 1.2) circle[radius=0.1];
\fill[black] (2.4, 1.2) circle[radius=0.1];
\fill[black] (1.2, 2.1) circle[radius=0.1];
\draw[thick] (0.6, 0) -- (0, 1.2);
\draw[thick] (1.8, 0) -- (2.4, 1.2);
\draw[thick, dashed] (0.6, 0) -- (1.8, 0);
\end{tikzpicture}
\caption{}
\end{subfigure}
\hfill
\begin{subfigure}{0.24\linewidth}
\centering
\begin{tikzpicture}
\fill[black] (0.6, 0) circle[radius=0.1];
\fill[black] (1.8, 0) circle[radius=0.1];
\fill[black] (0, 1.2) circle[radius=0.1];
\fill[black] (2.4, 1.2) circle[radius=0.1];
\fill[black] (1.2, 2.1) circle[radius=0.1];
\draw[thick] (0.6, 0) -- (0, 1.2);
\draw[thick] (1.8, 0) -- (2.4, 1.2);
\draw[thick, dashed] (0, 1.2) -- (2.4, 1.2);
\end{tikzpicture}
\caption{}
\end{subfigure}
\hfill
\begin{subfigure}{0.24\linewidth}
\centering
\begin{tikzpicture}
\fill[black] (0.6, 0) circle[radius=0.1];
\fill[black] (1.8, 0) circle[radius=0.1];
\fill[black] (0, 1.2) circle[radius=0.1];
\fill[black] (2.4, 1.2) circle[radius=0.1];
\fill[black] (1.2, 2.1) circle[radius=0.1];
\draw[thick] (0.6, 0) -- (0, 1.2);
\draw[thick] (1.8, 0) -- (2.4, 1.2);
\draw[thick, dashed] (0.6, 0) -- (2.4, 1.2);
\end{tikzpicture}
\caption{}
\end{subfigure}
\hfill
\begin{subfigure}{0.24\linewidth}
\centering
\begin{tikzpicture}
\fill[black] (0.6, 0) circle[radius=0.1];
\fill[black] (1.8, 0) circle[radius=0.1];
\fill[black] (0, 1.2) circle[radius=0.1];
\fill[black] (2.4, 1.2) circle[radius=0.1];
\fill[black] (1.2, 2.1) circle[radius=0.1];
\draw[thick] (0.6, 0) -- (0, 1.2);
\draw[thick] (1.8, 0) -- (2.4, 1.2);
\draw[thick, dashed] (1.8, 0) -- (0, 1.2);
\end{tikzpicture}
\caption{}
\end{subfigure}
\caption{Four paths containing two nonadjacent edges\label{fig-kp3-nonadjacent-blocks}}
\end{figure}
\end{example}

\newpage 
\begin{thm} \label{thm-kc-designs}
Let $E$ be the set of edges of $K_n$ and let $B$ be the set of cyclic subgraphs of $K_n$ with $\hat{k}$ vertices, for $\hat{k} \geq 3$. If $n = 2\hat{k} - 3$, the pair $(E, B)$ forms a balanced incomplete block design, denoted a KC design, with parameters
\begin{align*}
v & = {2 \hat{k} - 3 \choose 2}, \\
b & = \frac{(2 \hat{k} - 3)!}{2 \hat{k} (\hat{k} - 3)!}, \\
r & = \frac{(2 \hat{k} - 5)!}{(\hat{k} - 3)!}, \\
k & = \hat{k}, \\
\lambda & = \frac{(2 \hat{k} - 6)!}{(\hat{k} - 3)!}.
\end{align*}
\end{thm}

\begin{proof}
The fact that the blocks have size $k = \hat{k}$ is required explicitly by the construction. There are $n$ vertices of $K_n$ and each pair of them form an edge, so
\[ v = {n \choose 2} = {2 \hat{k} - 3 \choose 2}. \]

For the number of blocks, we begin by counting permutations of the $n = 2 \hat{k} - 3$ vertices into strings of length $\hat{k}$. However, every cyclic permutation of a string corresponds to the same cycle, and cycles are the same forwards as backwards, so we divide the number of permutations by $2 \hat{k}$ to obtain
\[ b = \frac{(2 \hat{k} - 3)!}{2 \hat{k} (\hat{k} - 3)!}. \]

For the replication number $r$, fix an edge $(a, b)$. There are $2 \hat{k} - 5$ vertices remaining, and $\hat{k} - 2$ more vertices are required to complete a block. By requiring that $a$ and $b$ be the ``first'' two vertices of the cycle, we avoid the overcounting encountered above, so the replication number is simply
\[ \frac{(2 \hat{k} - 5)!}{(\hat{k} - 3)!}. \]

As noted in the discussion of KP designs, every pair of edges in $K_n$ can be classified as either adjacent or nonadjacent. Thus it suffices to calculate the indices $\lambda_{adj}$ for a pair of adjacent edges and $\lambda_{non}$ for a pair of nonadjacent edges, and then show the two values to be equal. The counting argument for a pair of adjacent edges is nearly identical to the argument for $r$. Three vertices have been fixed, so we permute $\hat{k} - 3$ of the remaining $2 \hat{k} - 6$ vertices to complete a block. To count the blocks shared by a pair of nonadjacent edges, let $(a, b)$ and $(c, d)$ be two nonadjacent edges. Every block containing both $(a, b)$ and $(c, d)$ is a string of length $\hat{k}$ of the form
\[ a  \, b * * \cdots * c \, d * \cdots * \quad \text{or} \quad a \, b * * \cdots * d \, c * \cdots *, \]
in which each $*$ represents some other vertex. We must consider the relative order of $c$ and $d$ but not $a$ and $b$ because, for example, the strings $a b c d$ and $a b d c$ do not represent the same cycle, but the strings $a b c d$ and $b a d c$ do represent the same cycle. Since four vertices have been chosen, there are
\[ \frac{(2 \hat{k} - 7)!}{(\hat{k} - 3)!} \]
ways to fill in all of the asterisks. Furthermore, there are $\hat{k} - 3$ options for the position of $c d$ or $d c$ in the string, so altogether,
\[ \lambda_{non} = 2 (\hat{k} - 3) \frac{(2\hat{k} - 7)!}{(\hat{k} - 3)!} = \frac{(2 \hat{k} - 6) (2 \hat{k} - 7)!}{(\hat{k} - 3)!} = \lambda_{adj}. \]
Thus the design is indeed balanced.
\end{proof}

The KC designs are named analogously to the KP designs: K for ``Komplete'' and C for Cycle.

\begin{example} \label{example-cycle-design}
To parallel the previous example, we look at the KC design with block size $k = 4$, which is also built on the $K_5$ graph. As before, $v = 10$. The remaining parameters are
\begin{align*}
b & = \frac{(2 \cdot 4 - 3)!}{2 \cdot 4 \cdot (4 - 3)!} = \frac{5!}{8} = 15, \\
r & = \frac{(2 \cdot 4 - 5)!}{(4 - 3)!} = \frac{3!}{1} = 6, \\
\lambda & = \frac{(2 \cdot 4 - 6)!}{(4 - 3)!} = \frac{2!}{1} = 2.
\end{align*}

\begin{figure}[h]
\begin{subfigure}{0.24\linewidth}
\centering
\begin{tikzpicture}[scale=.7]
\fill[black] (0.8, 0) circle[radius=0.1];
\fill[black] (2.4, 0) circle[radius=0.1];
\fill[black] (0, 1.6) circle[radius=0.1];
\fill[black] (3.2, 1.6) circle[radius=0.1];
\fill[black] (1.6, 2.8) circle[radius=0.1];
\draw[thick] (0.8, 0) -- (0, 1.6) -- (1.6, 2.8);
\draw[thick, dashed] (1.6, 2.8) -- (3.2, 1.6) -- (0.8, 0);
\end{tikzpicture}
\caption{}
\end{subfigure}
\begin{subfigure}{0.24\linewidth}
\centering
\begin{tikzpicture}[scale=.7]
\fill[black] (0.8, 0) circle[radius=0.1];
\fill[black] (2.4, 0) circle[radius=0.1];
\fill[black] (0, 1.6) circle[radius=0.1];
\fill[black] (3.2, 1.6) circle[radius=0.1];
\fill[black] (1.6, 2.8) circle[radius=0.1];
\draw[thick] (0.8, 0) -- (0, 1.6) -- (1.6, 2.8);
\draw[thick, dashed] (1.6, 2.8) -- (2.4, 0) -- (0.8, 0);
\end{tikzpicture}
\caption{}
\end{subfigure}
\begin{subfigure}{0.24\linewidth}
\centering
\begin{tikzpicture}[scale=.7]
\fill[black] (0.8, 0) circle[radius=0.1];
\fill[black] (2.4, 0) circle[radius=0.1];
\fill[black] (0, 1.6) circle[radius=0.1];
\fill[black] (3.2, 1.6) circle[radius=0.1];
\fill[black] (1.6, 2.8) circle[radius=0.1];
\draw[thick] (0.8, 0) -- (0, 1.6);
\draw[thick] (1.6, 2.8) -- (3.2, 1.6);
\draw[thick, dashed] (0.8, 0) -- (1.6, 2.8);
\draw[thick, dashed] (0, 1.6) -- (3.2, 1.6);
\end{tikzpicture}
\caption{}
\end{subfigure}
\begin{subfigure}{0.24\linewidth}
\centering
\begin{tikzpicture}[scale=.7]
\fill[black] (0.8, 0) circle[radius=0.1];
\fill[black] (2.4, 0) circle[radius=0.1];
\fill[black] (0, 1.6) circle[radius=0.1];
\fill[black] (3.2, 1.6) circle[radius=0.1];
\fill[black] (1.6, 2.8) circle[radius=0.1];
\draw[thick] (0.8, 0) -- (0, 1.6);
\draw[thick] (1.6, 2.8) -- (3.2, 1.6);
\draw[thick, dashed] (0.8, 0) -- (3.2, 1.6);
\draw[thick, dashed] (0, 1.6) -- (1.6, 2.8);
\end{tikzpicture}
\caption{}
\end{subfigure}
\caption{Two cycles containing two adjacent edges, and two containing two nonadjacent edges\label{fig-kc-example}}
\end{figure}
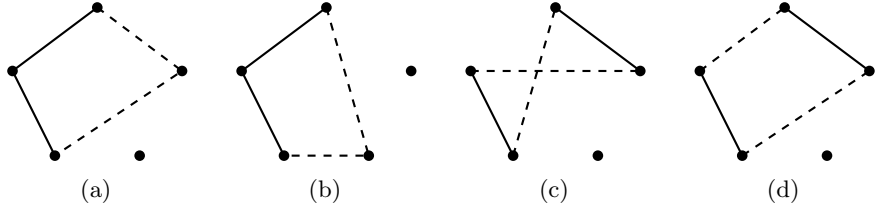
\end{example}

Figure \ref{fig-kc-example} shows the two blocks containing a pair of adjacent edges ((a) and (b)) and the two blocks containing a pair of nonadjacent edges ((c) and (d)).

The last two examples illustrate the similarity between KP designs and KC designs. We now make that connection rigorous within the language of design theory.

\begin{prop} \label{prop-exploded-designs}
Let $(V, B)$ be a block design with parameters $(v, b, r, k, \lambda)$. For $2 \leq j \leq k$, let $B'$ be the set of all $j$-subsets of elements of $B$. Then $(V, B')$ is a block design with parameters
\[ (v', b', r', k', \lambda') = \left( v, b {k \choose j}, r {k - 1 \choose j - 1}, j,  \lambda {k - 2 \choose j - 2} \right). \]
\end{prop}

\begin{proof}
Since $(V, B')$ is built upon the same variety set, $v' = v$. By design, blocks in $B'$ have size $k' = j$. Note that each block of $B$ produces $k \choose j$ blocks of $B'$, so $b' = b {k \choose j}$. Let $x \in V$. Then $x$ is in $r$ blocks of $B$. For each block in $B$, there are $k - 1 \choose j - 1$ blocks in $B'$ that contain $x$ since every set containing $x$ and $j - 1$ other varieties is a block. Similarly, let $y \in V$ with $y \neq x$. Then $x$ and $y$ are in $k - 2 \choose j - 2$ blocks in $B'$ for every block from $B$. Thus $r' = r {k - 1 \choose j - 1}$ and $\lambda' = \lambda {k - 2 \choose j - 2}$.
\end{proof}

Proposition \ref{prop-exploded-designs} is not new; for instance, it is mentioned in \cite{shramchenko}. However, these authors know of no terminology to describe this procedure, so we introduce our own: A block design constructed in the manner defined in Proposition \ref{prop-exploded-designs} is called an \emph{exploded design}. To emphasize the block size $j$ of the resulting design, we employ the phrase \emph{$j$-exploded design}. It is important to note that exploded designs can contain repeated blocks even if the original design did not.

\begin{thm} \label{thm-paths-cycles-exploded}
For every $k \geq 2$, the KP design with block size $k$ is the $k$-exploded design of the KC design with block size $k + 1$.
\end{thm}

\begin{proof}
Consider a path of size $k$. We can produce exactly one cycle of size $k + 1$ by adding an edge connecting the two endpoints of the path. Observe that any $k$-subset of a cycle of size $k + 1$ is a path of size $k$. Therefore, since the appropriate KC design contains every cycle of size $k + 1$ exactly once, its $k$-exploded design contains every path of size $k$ exactly once, which is the definition of the KP design with block size $k$.
\end{proof}

\begin{example}
We construct the 3-exploded design of the KC design in Example \ref{example-cycle-design}. The other parameters of the exploded design are
\begin{align*}
v' & = 10, \\
b' & = 15 \cdot {4 \choose 3} = 60, \\
r' & = 6 \cdot {3 \choose 2} = 18, \\
\lambda' & = 2 \cdot {2 \choose 1} = 4,
\end{align*}
which are indeed the parameters of the KP design of Example \ref{example-path-design}. Figure \ref{fig-exploded-paths-in-k5} shows the ${4 \choose 3} = 4$ exploded blocks obtained from one block of the KC design.

\begin{figure}[h]
\centering
\begin{subfigure}{0.24\linewidth}
\centering
\begin{tikzpicture}
\fill[black] (0.6, 0) circle[radius=0.1];
\fill[black] (1.8, 0) circle[radius=0.1];
\fill[black] (0, 1.2) circle[radius=0.1];
\fill[black] (2.4, 1.2) circle[radius=0.1];
\fill[black] (1.2, 2.1) circle[radius=0.1];
\draw[thick] (0.6, 0) -- (0, 1.2) -- (1.2, 2.1) -- (2.4, 1.2);
\draw[thick, dashed] (2.4, 1.2) -- (0.6, 0);
\end{tikzpicture}
\caption{}
\end{subfigure}
\hfill
\begin{subfigure}{0.24\linewidth}
\centering
\begin{tikzpicture}
\fill[black] (0.6, 0) circle[radius=0.1];
\fill[black] (1.8, 0) circle[radius=0.1];
\fill[black] (0, 1.2) circle[radius=0.1];
\fill[black] (2.4, 1.2) circle[radius=0.1];
\fill[black] (1.2, 2.1) circle[radius=0.1];
\draw[thick] (0, 1.2) -- (1.2, 2.1) -- (2.4, 1.2) -- (0.6, 0);
\draw[thick, dashed] (0.6, 0) -- (0, 1.2);
\end{tikzpicture}
\caption{}
\end{subfigure}
\hfill
\begin{subfigure}{0.24\linewidth}
\centering
\begin{tikzpicture}
\fill[black] (0.6, 0) circle[radius=0.1];
\fill[black] (1.8, 0) circle[radius=0.1];
\fill[black] (0, 1.2) circle[radius=0.1];
\fill[black] (2.4, 1.2) circle[radius=0.1];
\fill[black] (1.2, 2.1) circle[radius=0.1];
\draw[thick] (1.2, 2.1) -- (2.4, 1.2) -- (0.6, 0) -- (0, 1.2);
\draw[thick, dashed] (0, 1.2) -- (1.2, 2.1);
\end{tikzpicture}
\caption{}
\end{subfigure}
\hfill
\begin{subfigure}{0.24\linewidth}
\centering
\begin{tikzpicture}
\fill[black] (0.6, 0) circle[radius=0.1];
\fill[black] (1.8, 0) circle[radius=0.1];
\fill[black] (0, 1.2) circle[radius=0.1];
\fill[black] (2.4, 1.2) circle[radius=0.1];
\fill[black] (1.2, 2.1) circle[radius=0.1];
\draw[thick] (2.4, 1.2) -- (0.6, 0) -- (0, 1.2) -- (1.2, 2.1);
\draw[thick, dashed] (1.2, 2.1) -- (2.4, 1.2);
\end{tikzpicture}
\caption{}
\end{subfigure}
\caption{Four paths obtained from one cycle\label{fig-exploded-paths-in-k5}}
\end{figure}
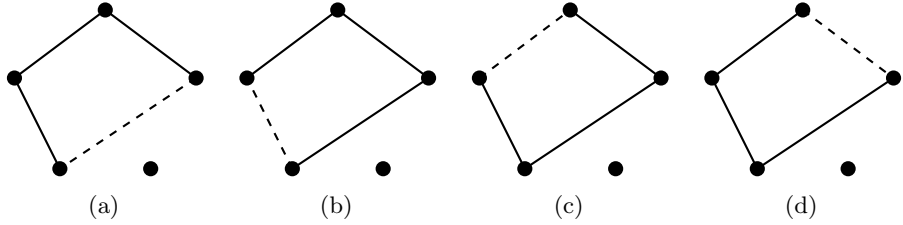
\end{example}

We have shown that any complete graph with $n$ vertices, $n$ odd and $n\geq 5$, yields two designs for which the edges of the graph serves as varieties. The KP design takes as its blocks the edge sets of path subgraphs, and the KC design likewise for cycle subgraphs. Furthermore, these two design are closely related, as the KP design is the exploded design of the KC design on the same underlying complete graph.

\section{Further Study}

The KP and KC designs grew out of an effort to generalize the complete graph designs described in Example 19.2 and Problem 19A of \cite{lint-wilson}. Van Lint and Wilson's Example 19.2 is Proposition \ref{prop-cam-lint-ex2.2} in this document. We were unsuccessful in our attempt to uncover a single construction for all of those designs, but the existence of the KP and KC designs sustains our hope that such a construction exists. The number of possible subgraphs to consider becomes unwieldy even for moderately-sized complete graphs, so in this work we limited our study to simple subgraphs like paths and cycles. Therefore, we expect there could be many more designs in the same vein. Perhaps one could benefit from some computer programming to expedite the enumeration of subgraphs.

The relationship between the KP designs and the KC designs raises questions about relationships between other designs. Of course, one can construct an exploded design from any block design, but of greater interest is the inverse problem: What existing designs are exploded designs of other designs? The graph-theoretic structure of the KP designs brought that question to the fore, but perhaps other designs are exploded designs as well.

On the topic of exploded designs, it may be worthwhile to determine some conditions under which exploded designs do or do not contain repeated blocks. For example, we showed that the $(k-1)$-exploded design of a KC design is the corresponding KP design, which has no repeated blocks, but the $(k - 2)$-exploded design of the same KC design does contain repeated blocks.

\newpage

\nocite{godsil-royle}

\printbibliography

\end{document}